\documentclass[article,12pt]{amsart}
\usepackage{mathpazo}
\usepackage{overpic}
\usepackage{amsmath,amsthm,amssymb,amscd,amsfonts,amsbsy}
\usepackage{color}
\usepackage{hyperref,url}
\usepackage{float}
\usepackage{hyperref}
\usepackage{enumerate}
\usepackage{enumitem}   
\usepackage{mathrsfs}  
\usepackage{mathrsfs}  
\usepackage[margin=2.2cm]{geometry}
\usepackage{xfrac}
\usepackage{amsthm}
\usepackage{amsmath}
\usepackage{amsfonts}
\usepackage{amssymb}
\usepackage{amsthm}
\usepackage{amsmath}
\usepackage{amsfonts}
\usepackage{amssymb}
\usepackage{mathtools}
\usepackage{mathrsfs}
\usepackage{comment}
\usepackage{soul}

\usepackage{amsmath, amssymb, graphics, setspace}
\usepackage[utf8]{inputenc}
\usepackage{listings,xcolor}

\newcommand{\mathsym}[1]{{}}
\newcommand{\unicode}[1]{{}}

\newcommand{\R}{\ensuremath{\mathbb{R}}}
\newcommand{\N}{\ensuremath{\mathbb{N}}}

\newcommand{\Z}{\ensuremath{\mathbb{Z}}}

\newcommand{\CC}{\mathcal{C}}

\newcommand{\CO}{\ensuremath{\mathcal{O}}}
\newcommand{\ov}{\overline}

\newcommand{\T}{\theta}

\newcommand{\s}{\ensuremath{\mathbb{S}}}

\newcommand{\x}{\mathbf{x}}

\newcommand{\z}{\mathbf{z}}
\newcommand{\ab}{\mathbf{a}}
\newcommand{\sgn}{\mathrm{sign}}

\newcommand{\taupm}{\tau^{\pm}}
\newcommand{\varphipm}{\varphi^{\pm}}
\newcommand{\Gammapm}{\Gamma^{\pm}}

\newcommand{\phipm}{\varphi^{\pm}}

\renewcommand{\d}{{\rm d}}

\def\p{\partial}
\def\e{\varepsilon}

\newtheorem {theorem} {Theorem}

\newtheorem {proposition} [theorem]{Proposition}

\newtheorem {lemma}  [theorem]{Lemma}

\newtheorem {remark} [theorem]{Remark}

\newtheorem {mtheorem} {Theorem}

\newcommand\blfootnote[1]{%
	\begingroup
	\renewcommand\thefootnote{}\footnote{#1}%
	\addtocounter{footnote}{-1}%
	\endgroup
}

\usepackage{mathpazo}

\begin{document}

\renewcommand{\arraystretch}{1.5}

\title[A Melnikov analysis on a family of second order discontinuous differential equations]
{A Melnikov analysis on a family of second order discontinuous differential equations}

\author[D.D. Novaes]
{Douglas D. Novaes}

\author[L. V. M. F. Silva]
{Luan V. M. F. Silva$ ^{*} $}

\address{Universidade Estadual de Campinas (UNICAMP), Departamento de Matem\'{a}tica, Instituto de Matem\'{a}tica, Estat\'{i}stica e Computa\c{c}\~{a}o Cient\'{i}fica (IMECC), Rua S\'{e}rgio Buarque de Holanda, 651, Cidade
Universit\'{a}ria Zeferino Vaz, 13083--859, Campinas, SP, Brazil}
\email{ddnovaes@unicamp.br}
\email{luanmattos@ime.unicamp.br}


\dedicatory{Dedicated to Jorge Sotomayor (in memoriam)}

\maketitle

\begin{abstract}
This paper aims to provide a Melnikov-like function that governs the existence of periodic solutions bifurcating from period annuli in certain families of second-order discontinuous differential equations of the form $\ddot{x}+\alpha\; \sgn(x)=\eta x+\e \;f(t,x,\dot{x})$. This family has attracted considerable attention from researchers, particularly in the analysis of specific instances of $f(t,x,\dot{x})$. The study of this type of differential equation is motivated by its significance in modeling systems with abrupt state changes, both in natural and engineering contexts.
\end{abstract}

\blfootnote{$^*$Corresponding author: Luan V. M. F. Silva, luanmattos@ime.unicamp.br} 

\section{Introduction and statement of the main result}\label{intro}
This paper is dedicated to the investigation of periodic solutions for a specific family of second-order discontinuous differential equations given by
\begin{equation}\label{s1}
\ddot{x}+\alpha\; \sgn(x)=\eta x+\e f(t,x,\dot{x}),
\end{equation}
where $\alpha$, $\eta$, and $\e$ are real parameters, and $ f $ is a $ \mathcal{C}^{1} $ function that is $\sigma$-periodic in the variable $t$. This class of differential equations has been the subject of extensive research in the past decade, primarily due to its relevance in various engineering phenomena characterized by abrupt state changes. For instance, in electronic systems, equation \eqref{s1} can describe an oscillator in the presence of a relay \cite{Jacquemard2012}, while in mechanical systems, it can model an automatic pilot for ships \cite{andronov}. Throughout this paper, we assume that $\e$ is a small parameter. Consequently, the differential equation \eqref{s1} for $\e=0$ is referred to as the unperturbed differential equation.

The investigation of periodic solutions for specific cases of the differential equation~\eqref{s1} has been a topic of interest in the research literature. For example, in \cite{Jacquemard2012}, periodic solutions of \eqref{s1} have been analyzed under the assumption of $\eta=0$ and $f(t,x,\dot{x})=\sin(t)$, while in \cite{Silva2020}, the focus was on the damped equation, i.e., $\eta\neq0$. The authors in \cite{Jiang2019} explored the existence of periodic solutions of \eqref{s1} using topological methods, employing a generalized version of the Poincaré-Birkhoff Theorem for this purpose. Periodic orbits were also studied for a specific instance of equation \eqref{s1} in \cite{Kunze1997}, where $f(t, x, \dot{x})=p(t)$ is assumed to be periodic and of class $\mathcal{C}^{6}$. A similar approach is taken in \cite{NovLSilva2024}, but with the assumption that $\eta=0$ and $p(t)$ is a Lebesgue integrable periodic function with a vanishing average. In both of the latter works, the primary objective was to establish the boundedness of all solutions of the respective differential equations.

The Melnikov method \cite{Melnikov1963} is a fundamental tool for determining the persistence of periodic solutions in planar smooth differential systems under non-autonomous periodic perturbations. It involves constructing a bifurcation function known as the {\it Melnikov function}, which has its simple zeros corresponding to periodic solutions that bifurcate from a period annulus of the differential system. The Melnikov function is obtained by expanding a Poincar\'{e} map, typically the time-$T$-stroboscopic map, into a Taylor series. In the smooth context, this map inherits the regularity of the flow.

The Melnikov analysis has been employed in investigating the existence of crossing periodic solutions in non-smooth differential systems, as demonstrated in \cite{Andrade2021,Bastos2019,Granados2012,NOVAES2022133523,Novaes2020}. Following the direction set by these references, our primary objective is to derive an explicit expression, using a Melnikov procedure, for a function that governs the existence of periodic solutions in \eqref{s1} when the unperturbed equation exhibits a period annulus.

Due to the discontinuous nature of \eqref{s1}, verifying the regularity of the time-$\sigma$ stroboscopic map associated with it presents challenges. To address this, we introduce time as a variable and utilize the discontinuous set generated by the sign function as a Poincaré section. This approach enables the construction of a smooth displacement function. An analogous approach has been previously explored by J. Sotomayor in his thesis \cite{Sotomayor1964} for autonomous differential equations. This function quantifies the distance between the positive forward flow and the negative backward flow where both intersect the discontinuous set. Subsequently, a Melnikov-like function is obtained by expanding this displacement function into a Taylor series.

The remainder of this section is dedicated to exploring the general notion of solutions of the differential equation \eqref{s1} (see Section \ref{sec:filippov}), classifying the unperturbed differential equation (see Section \ref{unpanalise}), and finally presenting the statement of our main result (see Section \ref{sec:main}). Additionally, an application of our main result is provided (see Section \ref{sec:example}).

\subsection{Filippov solutions}\label{sec:filippov}
To better understand the meaning of a solution to the differential equation \eqref{s1}, we make a variable change by defining $\dot{x}=y$. This change results in the following first order differential system
\begin{equation}\label{sis:s1}
X_\e(t,x,y):	\begin{cases}
		\dot{x}=y,\\
		\dot{y}=\eta  x-\alpha \sgn(x)+\e f(t,x,y).
	\end{cases}
\end{equation}
The solutions of the differential system \eqref{sis:s1} are defined according to the Filippov convention (see  \cite[\S 7]{Filippov1988}), which exist for every initial condition. For this reason, we will refer to \eqref{sis:s1} as a Filippov system. Consequently, the solutions of \eqref{s1} are derived from the solutions of the Filippov system \eqref{sis:s1}, ensuring existence for all possible initial conditions.

Taking into account the sign function present in \eqref{sis:s1}, we can decompose the differential system into the following ones:
\begin{equation}\label{sis:s3}
\begin{cases*}
		\dot{x}=y,\\
		\dot{y}=\eta  x-\alpha +\e f(t,x,y),
\end{cases*}
\quad x\geq0,\quad \text{and} \quad 
\begin{cases*}
	\dot{x}=y,\\
	\dot{y}=\eta  x+\alpha +\e f(t,x,y),
\end{cases*}
\quad x\leq0.
\end{equation}
Each one of the differential systems presented in \eqref{sis:s3} are of class $\mathcal{C}^{1}$ within its respective domain of definition, ensuring the uniqueness of solutions for each differential system in \eqref{sis:s3}. We shall focus our attention on solutions of the differential systems in \eqref{sis:s3} that intersect the region of discontinuity transversely. Under these conditions, solutions of \eqref{sis:s1} are obtained by concatenating solutions of \eqref{sis:s3}, which establishes the global uniqueness property in such cases. A more detailed discussion on this fact is postponed to Section \ref{proof}, and we suggest \cite{Guardia2011} for further topics on this theory.

\subsection{Analysis of the unperturbed Filippov system}\label{unpanalise}
 Before presenting our main results, we provide an analysis of the 
 unperturbed Filippov system, that is, for $ \e=0 $: 
\begin{equation*}\label{sis:unp}
X_0(t,x,y)=X_0(x,y):\begin{cases}
		\dot{x}=y, \\
		\dot{y}=\eta x-\alpha \sgn (x),
	\end{cases}
\end{equation*}
which matches 
\[
X^+(x,y):\left\{ \begin{array}{l}
	\dot{x}=y, \\
	\dot{y}=\eta x-\alpha, 	
\end{array} \right. \;\;\text{and} \;\; X^-(x,y):\left\{ \begin{array}{l}
	\dot{x}=y,\\
	\dot{y}=\eta x+\alpha,	
\end{array} \right.
\]
when restricted to $ x\geq0 $ and $ x\leq 0 $, respectively. We notice that, when $\alpha\neq0$, the line $\Sigma=\{(x,y)\in \R^2 / x=0 \}$ represents the set of discontinuity of $X_0$. Moreover, except for $ y=0 $, the line $ \Sigma $ corresponds to a crossing region of $ X_0$, meaning that the trajectories of $ X_0 $ are formed by concatenating the trajectories of $ X^{+} $ and $ X^{-} $ along $ \Sigma $. Additionally,  if we consider the involution $R(x,y)=(-x,y)$, we notice that $X^-(x,y)=-RX^+(R(x,y))$  and $ \text{Fix}(R)\subset \Sigma $, which means that the Filippov vector field $X_0$ is $R$-\text{reversible}. In the classical sense, a smooth planar vector field $X$ is said to be $R$-\text{reversible} if it satisfies $ X(x,y)=-R(X(R(x,y))) $.  The geometric meaning of this property is that the phase portrait of $X_0$ is symmetric with respect to $ \text{Fix}(R) $. Furthermore, by considering the involution $S(x,y)=(x,-y)$, we verify that both $X^{+}$ and $X^{-}$ are $S$-reversible in the classical sense. This indicates that the trajectories of $X^{+}$ and $X^{-}$ exhibit symmetry with respect to $\text{Fix}(S)=\{(x,y)\in\R^{2}:y=0\}$ (see Figures \ref{AA}, \ref{BB}, and \ref{CC}).

\begin{remark}\label{rmk:A}
	Let us denote $ \x=(x,y) $ and $ \ab=(0,\alpha) $. Then, taking $ A=\begin{psmallmatrix}
		0 & 1\\
		\eta & 0
	\end{psmallmatrix} $, the vector fields $ X^{+} $ and $ X^{-} $ can be rewritten as follows 
	\[
	X^{+}(\x)=A\x- \ab\quad \text{and} \quad X^{-}(\x)=A\x+\ab,
	\]
	for $ \x\in\Sigma^{+}=\{x>0\} $ and $ \x\in\Sigma^{-} =\{x<0\}$, respectively. By denoting the solutions of $ X^{+} $ and $ X^{-} $, as $ \Gamma^{+}(t,\z_0) $ and $ \Gamma^{-}(t,\z_0) $, respectively, with initial conditions $ \z_0\in\Sigma^{+} $ and $ \z_0\in\Sigma^{-} $, respectively, we notice that both solutions can be explicitly expressed by means of the variation of parameters as follows
\begin{equation*}
		\Gamma^{+}(t,\z_0)=e^{At}(\z_0-\int_{0}^{t}e^{-As}\cdot\ab\;\d s) \quad \text{and} \quad  \Gamma^{-}(t,\z_0)=e^{At}(\z_0+\int_{0}^{t}e^{-As}\cdot\ab\;\d s),
\end{equation*}
	where $ e^{At} $ is the exponential matrix of $ At $ given by
\begin{equation}\label{expA}
 e^{At} =\left\{\begin{aligned}
\begin{psmallmatrix}
	\cosh(t \sqrt{\eta})& \frac{\sinh (t \sqrt{\eta})}{\sqrt{\eta}} \\
	\sqrt{\eta} \sinh(t \sqrt{\eta})& \cosh(t \sqrt{\eta}) 
\end{psmallmatrix}  \quad &\text{if} \quad \eta>0,\vspace{2cm}\\
\begin{psmallmatrix}
	1&t\\0&1
\end{psmallmatrix}\quad &\text{if} \quad \eta=0,\vspace{2cm}\\
	\begin{psmallmatrix}
	\cos(t \sqrt{-\eta})& \frac{\sin (t \sqrt{-\eta})}{\sqrt{-\eta}} \\
	-\sqrt{-\eta} \sin(t \sqrt{-\eta})& \cos(t \sqrt{-\eta})
\end{psmallmatrix} \quad &\text{if} \quad \eta<0.
\end{aligned}\right.
\end{equation} 

The $ R $-reversibility of $ X_0 $ implies that  ${\Gamma^-(t,\z_0)=R(\Gamma^+(-t,R\cdot \z_0))}$ and
	\begin{equation}\label{eq:eA}
		e^{At}=R e^{-At}R.
	\end{equation}
\end{remark}

Now let us consider $ \x_0=(0,y_0) $, with $ y_0>0 $. For the sake of simplicity, we denote by $\Gamma^{+}(t,y_0)=(\Gamma_1^{+}(t,y_0),\Gamma_2^{+}(t,y_0))$ and $\Gamma^{-}(t,y_0)=(\Gamma_1^{-}(t,y_0),\Gamma_2^{-}(t,y_0))$ the solutions of $X^{+}$ and $X^{-}$, respectively, having $\x_0$ as the initial condition. Taking Remark \ref{rmk:A} and the expression for $ e^{At} $ given in \eqref{expA}, the solutions of $X^{+}$ having $ \x_0 $ as initial condition are given by

\begin{equation}\label{solucaoGama}
		\Gamma(t,y_0)= \left\{\begin{array}{l l}
			{\small\left(\dfrac{\alpha(1-\cosh \left(\sqrt{\eta }
				t\right))+\sqrt{\eta } y_0 \sinh
				\left(\sqrt{\eta } t\right)}{\eta },y_0 \cosh
			\left(\sqrt{\eta } t\right)-\dfrac{\alpha  \sinh
				\left(\sqrt{\eta } t\right)}{\sqrt{\eta }}\right)}, & \eta>0, \vspace{0.3cm}\\
			{\small\left(t y_0-\frac{\alpha  t^2}{2},y_0-\alpha 
			t\right)},& \eta=0,\vspace{0.3cm}\\
			{\small\left(\dfrac{\alpha( \cos (\omega t  )-1)+\omega  y_0
				\sin (\omega t )}{\omega^{2} },y_0 \cos (\omega t 
			)-\dfrac{\alpha  \sin ( \omega t )}{\omega }\right)}, &\eta<0,
		\end{array}\right.
\end{equation}
where $ \omega=\sqrt{-\eta} $, for $ \eta<0 $. As previously mentioned, the $R$-reversibility of $X_0$ allows us to easily obtain the expressions of $\Gamma^{-}(t,y_0)$ just by considering the relation
\begin{equation}\label{rel}
\Gamma^{+}(t,y_0)=\Gamma(t,y_0)\quad \text{and}\quad 	\Gamma^-(t,y_0)=R(\Gamma^+(-t,y_0)).
\end{equation}

Next, we classify all possible geometric configurations that $X_0$ can take by varying the values of $\alpha$ and $\eta$. Additionally, we provide a detailed description of the singularities of $X_0$ within the context of Filippov systems (see \cite{Guardia2011} ). By denoting $ p=(0,0) $, $p^+=(\alpha/\eta,0)$, and $p^-=(-\alpha/\eta,0)$, for $ \eta\neq0 $, the geometric configurations of $ X_0 $ are as follows:
\begin{itemize}
	\item[\textbf{(C1)}] $\eta>0$ and $\alpha>0$:  In this case, the points $p$, $p^+$, and $p^-$ are the singularities of $X$, with $p$ being an invisible fold-fold point and also a center-type.  The points $p^+$ and $p^-$ are both admissible linear saddles (see Figure \ref{AA}); 
	
	\item[{\bf (C2)}] $\eta>0$ and $\alpha=0$: In this case, $ X_0 $ corresponds to a  smooth vector field having $p$ as the only singularity of saddle-type (see Figure \ref{AA}); 
	
	\item[\textbf{(C3)}] $\eta>0$ and $\alpha<0$: In this case, the only singularity of $X_0$ is $p$ and it is a visible fold-fold (see Figure \ref{AA}); 
	
	\item[\textbf{(C4)}] $\eta=0$ and $\alpha>0$: In this case, the only singularity of $ X_0 $ is $p$, which corresponds to an invisible fold-fold as well as a center (see Figure \ref{BB});
	
	\item[\textbf{(C5)}] $\eta=0$ and $\alpha=0$: In this case, $ X_0 $ is a smooth vector field with $\{y=0\}$ being the set of its critical points (see Figure \ref{BB});
	
	\item[\textbf{(C6)}] $\eta=0$ and $\alpha<0$: In this case, $p$ is a visible fold-fold of $X_0$ and also its only singularity (see Figure \ref{BB});
	
	\item[\textbf{(C7)}] $\eta<0$ and $\alpha>0$: In this case, the point $p$ is an invisible fold-fold and also a center of $ X_0 $ (see Figure \ref{CC});
	
	\item[\textbf{(C8)}] $\eta<0$ and $\alpha=0$: In this case, $ X_0 $ is a smooth vector field and it has $p$ as the only singularity being a center-type (see Figure \ref{CC});
	
	\item[\textbf{(C9)}] $\eta<0$ and $\alpha<0$: In this case, the point $p$ is a visible fold-fold, and the points $p^+$ and $p^-$ are both linear centers of $ X_0 $. These points are the only singularities of $ X_0 $ (see Figure \ref{CC}).
\end{itemize}

In the cases \textbf{(C1)}, \textbf{(C4)}, \textbf{(C7)}, \textbf{(C8)}, and \textbf{(C9)}, we notice the existence of a region composed by a family of periodic orbits (see Figures \ref{AA}, \ref{BB}, and \ref{CC}). This region is called \textit{period annulus} and it will be denoted by $\mathcal{A}_i$, with $i\in\{1,4,7,8,9\}.$ For each one of these cases, there exist a half--period function, $ \tau^{+}(y_0)$ (resp. $ \tau^{-}(y_0)$ ), for $ y_0>0 $, providing the smallest (resp. greatest) time for which the solution $\Gamma^+(t,y_0)$ (resp. $\Gamma^-(t,y_0)$), with $(0,y_0)\in\mathcal{A}_i$, reaches the discontinuous line $\Sigma$ again. In each of these cases, the half--period function is given by $ \tau^{+}(y_0)= \tau_0(y_0)$ and $ \tau^{-}(y_0)= -\tau_0(y_0)$,  where, using the $ S $-reversibility of $ X^{+} $, the function $ \tau_0(y_0) $ is the solution of the boundary problem
\[
\begin{cases}
	\Gamma(0,y_0)=(0,y_0),\\
	\Gamma_2\left(\frac{\tau_0(y_0)}{2},y_0\right)=0.
\end{cases}
\] 
The expression  for $ \tau_0 $ is given by 
 \begin{equation}\label{tauzero}
		\tau_0:\left\{ \begin{array}{l l}
			(0,\frac{\alpha}{\sqrt{\eta}})\to(0,\infty) & \mbox{\bf(C1)},\vspace{0.3cm}\\
			(0,\infty)\to(0,\infty) & \mbox{\bf(C4)},\vspace{0.3cm}\\
			(0,\infty)\to (0,\frac{\pi}{\omega}) & \mbox{\bf(C7)}, \vspace{0.3cm}\\
			(0,\infty)\to \{\frac{\pi}{\omega}\} & \mbox{\bf(C8)}, \vspace{0.3cm}\\
			(0,\infty)\to (\frac{\pi}{\omega},\frac{2\pi}{\omega}) & \mbox{\bf(C9)},
		\end{array}\right. \;\; \mbox{and} \;\; \tau_0(y_0)=\left\{ \begin{array}{l l}
			\tfrac{1}{\sqrt{\eta}}\log\left(\tfrac{\alpha+y_0\sqrt{\eta}}{\alpha-y_0\sqrt{\eta}}\right) & \mbox{\bf(C1)},\vspace{0.3cm}\\
			\frac{2y_0}{\alpha} & \mbox{\bf(C4)},\vspace{0.3cm}\\
			\tfrac{2}{\omega}\arctan\left(\tfrac{\omega y_0}{\alpha}\right) & \mbox{\bf(C7)}, \vspace{0.3cm}\\
			\frac{\pi}{\omega}& \mbox{\bf(C8)}, \vspace{0.3cm}\\
			\tfrac{2}{\omega}\left(\pi+\arctan\left(\tfrac{\omega y_0}{\alpha}\right)\right) & \mbox{\bf(C9)}.
		\end{array}\right.
\end{equation}

It is noteworthy that in cases \textbf{(C7)} and \textbf{(C9)} the boundary problem above provides infinitely many solutions. Nevertheless, due to the dynamics of the unperturbed Filippov system near $ y_0=0$, the solutions to be considered must satisfy the conditions $ \lim_{y_0\to0^{+}}\tau_0(y_0) =0$ in case  \textbf{(C7)} and $ \lim_{y_0\to0^{+}}\tau_0(y_0) =2\pi/\omega$ in case \textbf{(C9)}.

	Taking into account the $S$-reversibility of both $X^{+}$ and $X^{-}$, along with equation \eqref{rel}, we can deduce that
	\begin{equation}\label{eq:y0}
		\Gamma^-(-\tau_0(y_0),y_0)=\Gamma^+(\tau_0(y_0),y_0)=(0,-y_0),
	\end{equation}
	which will play an important role throughout this work.

\begin{remark}\label{rmk:C8}
	We notice that the origin $ p=(0,0) $ in case \textbf{(C8)} as well as the points $p^+=(\alpha/\eta,0)$ and $p^-=(-\alpha/\eta,0)$ in case \textbf{(C9)} correspond to linear centers, which have been consistently studied in the literature (see, for instance,  \cite{Novaes2013}). For this reason, from now on, we will not consider these centers in our study.
\end{remark}

For $i\in\{1,4,7,9\}$, we denote by $\mathcal{D}_i$, the interval of definition of $\tau_0$, and by $\mathcal{I}_i$ the image of $\tau_0$. From the possible expressions of $\tau_0$ in \eqref{tauzero}, it can be observed that $\tau'_0(y_0)>0$ in $\mathcal{D}_i$, for $i\in\{1,4,7\}$, and $\tau'_0(y_0)<0$ in $\mathcal{D}_9$.  Consequently, $\tau_0$ is a bijection between $\mathcal{D}_i$ and $\mathcal{I}_i$ for $i\in\{1,4,7,9\}$. Its inverse is given by

 \begin{equation}\label{inversetime}
 	v:\left\{ \begin{array}{l l}
 		(0,\infty)\to(0,\frac{\alpha}{\sqrt{\eta}}) & \mbox{\bf(C1)},\vspace{0.3cm}\\
 		(0,\infty)\to (0,\infty)& \mbox{\bf(C4)},\vspace{0.3cm}\\
 	(0,\frac{\pi}{\omega}) \to (0,\infty)& \mbox{\bf(C7)}, \vspace{0.3cm}\\
 	(\frac{\pi}{\omega},\frac{2\pi}{\omega})\to(0,\infty)& \mbox{\bf(C9)},
 	\end{array}\right. \;\; \mbox{and} \;\;
		v(\sigma)=\left\{\begin{array}{cl}
			\dfrac{\alpha}{\sqrt{\eta}}\tanh\left(\dfrac{\sqrt{\eta}\,\sigma}{2}\right)& \mbox{\textbf{(C1)}},\vspace{0.2cm}\\
			\dfrac{\alpha \,\sigma}{2}&  \mbox{\textbf{(C4)}},\vspace{0.2cm}\\
			\dfrac{\alpha}{\omega}\tan\left(\dfrac{\omega\,\sigma}{2}\right)&  \mbox{\textbf{(C7)}},\vspace{0.2cm}\\
			\dfrac{\alpha}{\omega}\tan\left(\dfrac{\omega\,\sigma}{2}\right)	&  \mbox{\textbf{(C9)}}.
		\end{array}\right.
\end{equation}
 Although the expressions for $ v(\sigma) $ coincide in cases \textbf{(C7)} and \textbf{(C9)}, the functions are not the same. Indeed, the domains do not coincide and $ \alpha $ assumes distinct signs in each case.

\begin{figure}[!htb]
	\begin{center}
		\begin{overpic}[scale=1.1, unit=0.1mm]{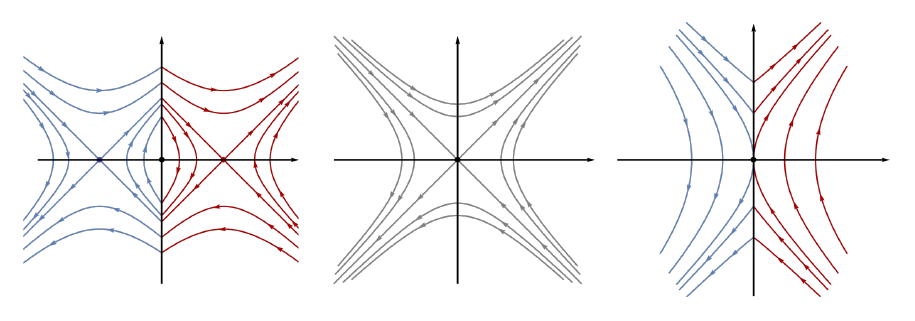}
			\put(15.5,-1){\small $\textbf{(C1)}$}
			\put(32.7,17){\small $x$}
			\put(17.2,32){\small $y$}
			\put(17.8,15.8){\small$p$}
			\put(10.3,19){\small $p^-$}
			\put(23.8,19){\small$p^+$}
			\put(48,-1){\small $\textbf{(C2)}$}
			\put(65.2,17){\small $x$}
			\put(49.6,32){\small $y$}
			\put(50.3,14.6){\small $p$}
			\put(80.5,-1){\small $\textbf{(C3)}$}
			\put(97.5,17){\small $x$}
			\put(82,33.5){\small $y$}
			\put(83,15.8){\small$p$}
		\end{overpic}		
	\end{center}
	\smallskip
	\caption{ Phase portraits of the cases \textbf{(C1)}, \textbf{(C2)}, and \textbf{(C3)}.}
	\label{AA}
\end{figure}
  
\begin{figure}[!htb]
	\begin{center}
		\begin{overpic}[scale=1.1, unit=0.1mm]{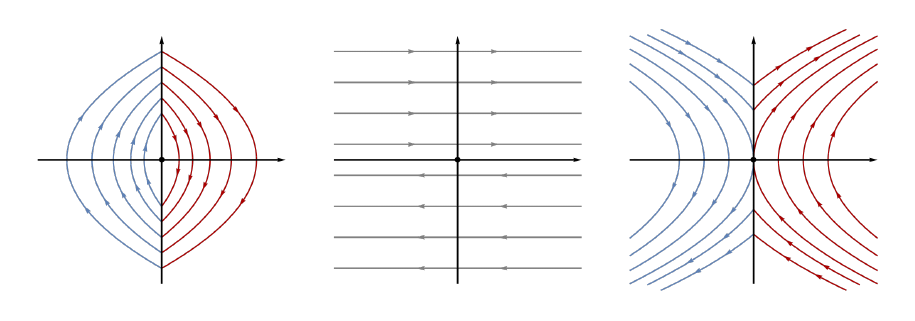}
			\put(15.8,0){\small $\textbf{(C4)}$}
			\put(31.4,17){\small  $x$}
			\put(17.2,32.2){\small  $y$}
			\put(17.9,16.1){\small  $p$}
			\put(48,0){\small $\textbf{(C5)}$}
			\put(63.9,17){\small  $x$}
			\put(49.5,32.2){\small  $y$}
			\put(80.8,0){\small  $\textbf{(C6)}$}
			\put(96.0,17){\small $x$}
			\put(81.8,32.2){\small  $y$}
			\put(82.8,16.1){\small  $p$}
		\end{overpic}		
	\end{center}
	\smallskip
	\caption{ Phase portraits of the cases \textbf{(C4)}, \textbf{(C5)}, and \textbf{(C6)}.}
	\label{BB}
\end{figure}

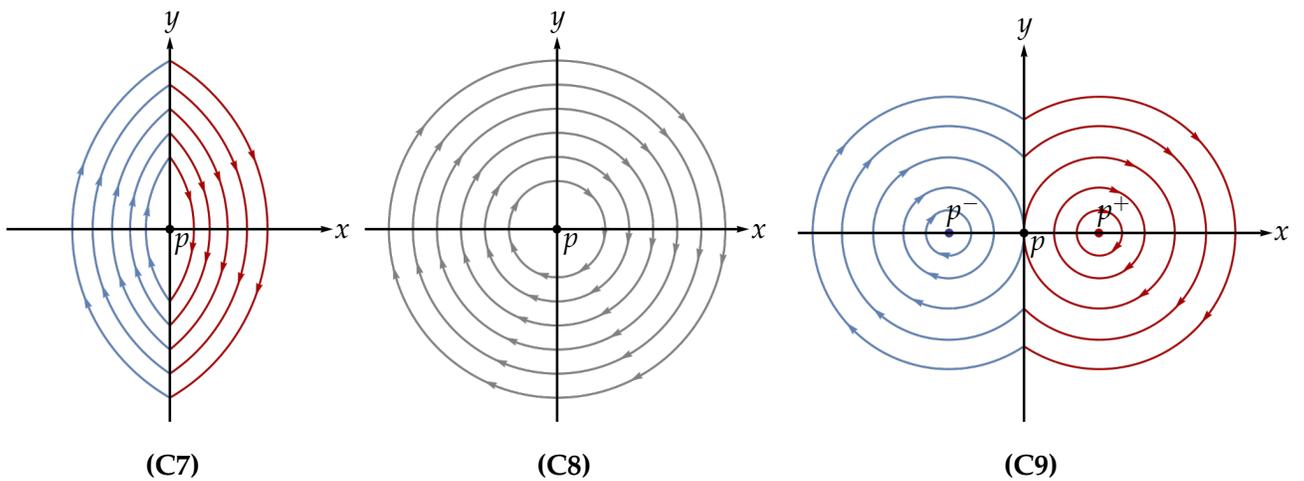
\begin{figure}[H]
	\begin{center}
		\begin{overpic}[scale=1.2, unit=0.1mm]{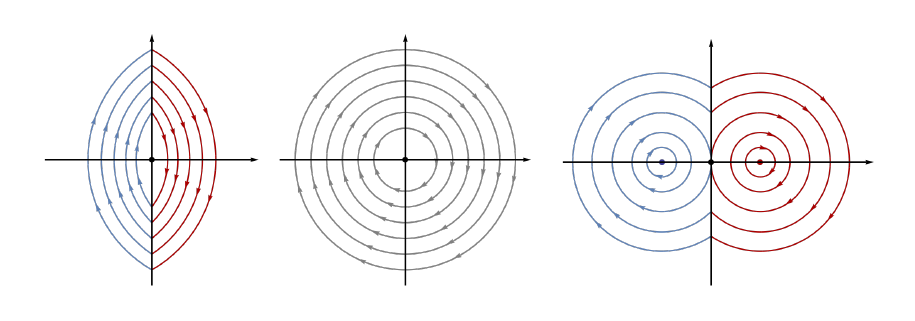}
			\put(14.8,0){\small $\textbf{(C7)}$}
			\put(28.3,17){\small  $x$}
			\put(16.2,32.2){\small $y$}
			\put(16.8,16.3){\small  $p$}
			\put(42.8,0){\small  $\textbf{(C8)}$}
			\put(58.1,17){\small  $x$}
			\put(43.8,32.2){\small  $y$}
			\put(44.6,16.3){\small $p$}
			\put(76.2,0){\small $\textbf{(C9)}$}
			\put(95.5,16.8){\small  $x$}
			\put(77.2,31.7){\small  $y$}
			\put(78,16.1){\small  $p$}
			\put(82.8,18.3){\small  $p^+$}
			\put(72,18.3){\small  $p^-$}
		\end{overpic}		
	\end{center}
	\smallskip
	\vspace{0.4cm}
	\caption{ Phase portraits of the cases \textbf{(C7)}, \textbf{(C8)}, and \textbf{(C9)}.}
	\label{CC}
\end{figure}

\subsection{Main result}\label{sec:main}

	As usual, the Melnikov method for determining the persistence of periodic solutions provides a bifurcation function whose simple zeros are associated with periodic solutions bifurcating from a period annulus. In what follows we are going to introduce this function for the differential equation \eqref{s1}. 
	
Recall that $ f(t,x,\dot{x}) $ is a $ \mathcal{C}^{1} $ function, $ \sigma- $periodic in the variable $ t $. Let $ \Gamma(t,y_0) $ and $ v(\sigma) $ be the functions defined in \eqref{solucaoGama} and \eqref{inversetime}, respectively. We define the Melnikov--like function  $M:\R\rightarrow\R$ as
\begin{equation}\label{eq:M}
	M(\phi)=\int_0^{\frac{\sigma}{2}}U\left(t,\frac{\sigma}{2}\right)\left(f\left(\phi+t,\Gamma\left(t,v\left(\frac{\sigma}{2}\right)\right)\right)+f\left(\phi-t,R\Gamma\left(t,v\left(\frac{\sigma}{2}\right)\right)\right)\right)\d t, 
\end{equation}
where  
\begin{equation}\label{eq:U}
	U(t,\sigma)=\left\{\begin{array}{cl}
	-\dfrac{\alpha}{\sqrt{\eta}}\,\textrm{sech}\left(\dfrac{\sqrt{\eta}\,\sigma}{2}\right)\sinh\left(\dfrac{\sqrt{\eta}\,(2t-\sigma)}{2}\right)& \eta>0,\vspace{0.2cm}\\
	-\dfrac{\alpha(2t-\sigma)}{2}& \eta=0,\vspace{0.2cm}\\
	-\dfrac{\alpha}{\omega}\,\sec\left(\dfrac{\omega\ \sigma}{2}\right)\sin\left(\dfrac{\omega\,(2t-\sigma)}{2}\right)& \eta<0,
\end{array}\right.
\end{equation}
with $  \omega=\sqrt{-\eta}>0,$ for $\eta<0$. Notice that, since $ f(t,x,\dot{x}) $ is $ \sigma $-periodic in $t $, the Melnikov--like function $ M $ is $ \sigma $-periodic. We derive the expression of $ M $ in the proof of the following   result concerning periodic solutions of the differential equation \eqref{s1}. Its proof is postponed to Section \ref{proof}.

\begin{mtheorem}\label{ta}
Suppose that for some $i\in\{1,4,7,9\}$ the parameters $ \alpha $	and $ \eta $ of the differential equation \eqref{s1} satisfy the condition \textbf{(Ci)} and that  $\sigma/2\in\mathcal{I}_i$. Then, for each $\phi^*\in[0,\sigma]$, such that $M(\phi^*)=0$ and $M'(\phi^*)\neq0$, there exists $\ov\e>0$ and a unique smooth branch $x_{\e}(t)$, $\e\in(-\ov \e,\ov \e)$, of isolated $\sigma$-periodic solutions of the differential equation \eqref{s1} satisfying  $(x_0(\phi^*),\dot{x}_0(\phi^*))=(0,v(\sigma/2))$.
\end{mtheorem}

\begin{remark}
It is important to emphasize that periodic solutions obtained in Theorem \ref{ta} bifurcate from the interior of the period annulus $ \mathcal{A}_i$ for $ i\in\{1,4,7,9\} $. In case \textbf{(C9)}, the two homoclinic orbits joining $ p=(0,0) $ constitute the boundary of $ \mathcal{A}_9 $. Therefore, their persistence are not being considered in Theorem \ref{ta}, which is elucidated by the conclusion $\dot x_0(\phi^*)= v(\sigma/2)\in \mathcal{D}_9=(0,\infty)$.
\end{remark}

\subsection{Example}\label{sec:example}
In order to illustrate the application of Theorem \ref{ta}, we examine the following differential equation 
\begin{equation}\label{exmel}
	\ddot{x}+\alpha\; \sgn(x)= \eta x+\e\; \sin(\beta\;t),\quad \text{with} \quad \alpha>0\quad  \text{and} \quad  \eta >0 .
\end{equation}
This equation was previously studied in \cite{Silva2020}, where the authors provided conditions on $ \alpha $, $ \eta $, and $ \beta  $ to determine the existence of a discrete family of simple periodic solutions of \eqref{exmel}. By assuming $ f(t,x,\dot{x})=\sin(\beta t) $ in \eqref{s1}, we reproduce a similar result as in \cite[Theorem 2.1.1]{Silva2020}, as follows. 

\begin{proposition}\label{P1}
	Given $ n\in\N $, there exists $ \ov\e _n>0 $ such that, for every $ \e\in(- \ov\e _n, \ov\e _n) $, the differential equation \eqref{exmel} has $n$ isolated $ \frac{2\pi}{\beta}(2k-1)$-periodic solutions, for $ k\in\{1,\ldots,n\} $,  whose initial conditions are  $(x_{\e}^{k}(0),\dot{x}_{\e}^{k}(0))=\left(0,\frac{\alpha}{\sqrt{\eta}}\tanh(\frac{\pi\sqrt{\eta}\,(2k-1)}{\beta}) \right)+\CO(\e)$.
\end{proposition}

The  main difference between both results is that Proposition \ref{P1} is based on perturbation theory, while Theorem 2.1.1 in \cite{Silva2020} provides a precise upper bound for $ \e $ by means of direct computations.

\begin{proof}
In equation \eqref{exmel}, the parameters $ \alpha $ and $ \eta  $ are satisfying condition \textbf{(C1)}. Let us define $ \sigma_i=2\pi\beta^{-1} i $, for $ i\in\N $. Since $ f(t,x,\dot{x})=\sin(\beta t) $ is $ \sigma_1$-periodic in $ t $, it is natural that $ f(t,x,\dot{x}) $ is also $\sigma_i$-periodic in $ t $. Besides that, for every $ i\in\N $, we have $ \sigma_i/2\in\mathcal{I}_1 $. Then, for $ i\in\N $, we compute the Melnikov function, defined as \eqref{eq:M}, corresponding to the period $ \sigma_i $. This function takes the form
\[
M_i(\phi)=\dfrac{2\alpha(1+(-1)^{i-1})\sin(\beta\,\phi)}{\beta^2+\eta}.
\]
Notice that, if $ i  $ is odd, then 
\[
M_i(\phi)=\dfrac{4\alpha\sin(\beta\,\phi)}{\beta^2+\eta},
\]
while for even values of $ i $, $ M_i(\phi)=0 $. Given $ n\in\N $, we observe that $M_{2k-1}(0)=0$ and $M_{2k-1}'(0)\neq0$ for each $  k\in\{1,\ldots,n\}  $. Applying Theorem \ref{ta}, it follows that, for each $  k\in\{1,\ldots,n\}  $, there exists $ \e_k>0 $ and a unique branch $ x_{\e}^{k}(t) $, $ \e\in(-\e_k,\e_k) $, of isolated $  \ov\sigma_k$-periodic solutions of the differential equation \eqref{exmel} satisfying  $(x_{\e}^{k}(0),\dot{x}_{\e}^{k}(0))=(0,\nu(\ov\sigma_k)/2 )+\CO(\e)$, where $ \ov\sigma_k=\sigma_{2k-1} $. We see that, for each $ k\in\{1,\ldots,n\} $, the periods $ \ov\sigma_k $ are pairwise distinct, indicating that $ \nu(\ov\sigma_{k_1}/2) \neq \nu(\ov\sigma_{k_2}/2) $ whenever $ k_1\neq k_2 $. Therefore, by considering $ \ov\e_n=\min_{1\leq k\leq n} \{\e_k\}$, we conclude the proof of the proposition.
\end{proof}

\section{Proof of Theorem \ref{ta}}\label{proof}

In order to prove Theorem \ref{ta}, we consider the extended differential system associated \eqref{sis:s1}, given by
\begin{equation}\label{eq:ext}
	\begin{cases*}
		\dot{\T}=1,\\
		\dot{x}=y,\\
		\dot{y}=\eta x-\alpha\sgn(x)+\e f(\T,x,y),
	\end{cases*}
\end{equation}
which, due to the $ \sigma $-periodicity of $ f(\T,x,y) $ in $ \T $, has $ \s_{\sigma}\times\R^{2} $, with $ \s_{\sigma}=\R/\sigma\Z $, as its extended phase space. The differential system \eqref{eq:ext} matches
\begin{equation}\label{sis:extpn}
	\begin{cases*}
		\dot{\T}=1,\\
		\dot{x}=y,\\
		\dot{y}=\eta x-\alpha+\e f(\T,x,y),
	\end{cases*}
	\text{and}\;\;
	\begin{cases*}
		\dot{\T}=1,\\
		\dot{x}=y,\\
		\dot{y}=\eta x+\alpha+\e f(\T,x,y),
	\end{cases*}
\end{equation}
when it restricted to $ x\geq0 $ and $ x\leq 0 $, respectively. The solutions for the differential systems in \eqref{sis:extpn} with initial condition $ (\T_0,0,y_0) $, for $ y_0 >0$, are given by the functions
\begin{equation}\label{PHIP}
	\Phi^{+}(\tau,\T_0,y_0;\varepsilon)=(\tau+\T_0,\varphi^{+}(\tau,\T_0,y_0;\e)),
\end{equation}
and 
\begin{equation}\label{PHIN}
	\Phi^{-}(\tau,\T_0,y_0;\varepsilon)=(\tau+\T_0,\varphi^{-}(\tau,\T_0,y_0;\e)),
\end{equation}
respectively, where $ \varphipm(\tau,\T_0,y_0;\e) =( \varphipm_1(\tau,\T_0,y_0;\e) , \varphipm_2(\tau,\T_0,y_0;\e) )$ is the solution for the Cauchy problem
\begin{equation}\label{sis:s2}
	\begin{cases*}
		\dot{\x}=A \x\mp\ab+\e F(t+\T_0,\x),\\
		\x(0)=\x_0,
	\end{cases*}
\end{equation}
with $ \x_0=(0,y_0)$ and $F(t,\x)=(0,f(t,\x))$. Since the switching plane $ \Sigma'=\{(\T,x,y)\in\R^{3}:x=0\} $, with $ y\neq0 $, is a crossing region for the differential system \eqref{eq:ext}, it follows that solutions of \eqref{eq:ext} arise from the concatenation of $ \Phi^{+} $ and $ \Phi^{-} $ along $ \Sigma' $ when these solutions intersect $ \Sigma' $ transversely, as previously mentioned. We denote this concatenated solution by $ \Phi(\tau,\T_0,y_0;\varepsilon) $. 

\subsection{Construction of the displacement function} Taking \eqref{eq:y0} into account, we notice that,\linebreak for $y_0 > 0$,
\[
\varphi_1^{\pm}(\tau^{\pm}_0(y_0),\T_0,y_0;0)=\Gamma^{\pm}_1(\tau^{\pm}_0(y_0),y_0)=0,
\]
and
\begin{equation}\label{eq:Gama2}
	\frac{\p \phipm_1}{\p t}(\tau^{\pm}_0(y_0),\T_0,y_0;0)=\frac{\p \Gamma^{\pm}_1}{\p t}(\tau^{\pm}_0(y_0),y_0)=\Gamma^{\pm}_2(\tau^{\pm}_0(y_0),y_0)=-y_0\neq0.
\end{equation}
According to the Implicit Function Theorem, there exist smooth functions $\tau^{+}:\mathcal{U}_{(\T_0,y_0;0)}\to\mathcal{V}^{+}_{\tau^{+}_0(y_0)}$ and $\tau^{-}:\mathcal{U}_{(\T_0,y_0;0)}\to\mathcal{V}^{-}_{\tau^{-}_0(y_0)}$, where $ \mathcal{U}_{(\T_0,y_0;0)} $ and $\mathcal{V}^{\pm}_{\taupm_0(y_0)} $ are small neighborhoods of $ (\T_0,y_0;0)  $ and $ \taupm_0(y_0) $, respectively. These functions satisfy 
\begin{equation}\label{ss}
\taupm(\T_0,y_0;0)=\taupm_0(y_0) \quad \text{and} \quad	\varphi_1^{\pm}(\taupm(\T,y;\e),\T,y;\e)=0,
\end{equation}
for every $(\T,y;\e)\in\mathcal{U}_{(\T_0,y_0;0)}$. 

Thus, the functions $ \tau^{+}(\T_0,y_0;\e) $ and $ \tau^{-}(\T_0,y_0;\e) $, combined with the solutions \eqref{PHIP} and \eqref{PHIN} of the extended differential system \eqref{eq:ext}, allow us to construct a displacement function, $ \Delta(\T_0,y_0;\e) $, that ``controls'' the existence of periodic solutions of  \eqref{s1} as follows
\begin{equation*}\label{disp}
	\begin{aligned}
		\Delta(\T_0,y_0;\e)&=\Phi^{+}(\tau^{+}(\T_0,y_0;\e),\T_0,y_0;\varepsilon)-\Phi^{-}(\tau^{-}(\T_0+\sigma,y_0;\e),\T_0+\sigma,y_0;\varepsilon)\\
		&=\Phi^{+}(\tau^{+}(\T_0,y_0;\e),\T_0,y_0;\varepsilon)-\Phi^{-}(\tau^{-}(\T_0,y_0;\e),\T_0+\sigma,y_0;\varepsilon),
	\end{aligned}
\end{equation*}
since $ \tau^{-}(\T_0,y_0;\e)  $ is $ \sigma $-periodic in $ \T_0 $. The displacement function $ \Delta(\T_0,y_0;\e) $ computes the difference in $ \Sigma' $ between the points $ \Phi^{+}(\tau^{+}(\T_0,y_0;\e),\T_0,y_0;\varepsilon) $ and $ \Phi^{-}(\tau^{-}(\T_0,y_0;\e),\T_0+\sigma,y_0;\varepsilon) $ (see Fig. \ref{fig}). Thus it is straightforward that if $ \Delta(\T_0^{*},y_0^{*};\e^{*})=0 $ , for some $ (\T^{*},y_0^{*};\e^{*})\in[0,\sigma]\times\R^{+}\times\R$, then the solution $\Phi(\tau,\T_0^{*},y_0^{*};\varepsilon^{*})   $ is $ \sigma $-periodic in $ \tau $, meaning that $ \Phi(\tau,\T_0^{*},y_0^{*};\varepsilon^{*})$ and $ \Phi(\tau+\sigma,\T_0^{*},y_0^{*};\varepsilon^{*})$ are identified in the quotient space $ \s_{\sigma}\times\R^{2} $. Furthermore, from the definition of $\Phi^{+}$ and $\Phi^{-}$ in \eqref{PHIP} and \eqref{PHIN}, respectively, we have that
\begin{align*}
	\Delta(\T_0,y_0;\e)&=(\Delta_1(\T_0,y_0;\e),0,\Delta_3(\T_0,y_0;\e))\\
	&:=(\tau^+(\T_0,y_0;\varepsilon)-\tau^-(\T_0,y_0;\varepsilon)-\sigma,0,\\
	&\hspace{0.9cm}\varphi^+_2(\tau^+(\T_0,y_0;\varepsilon),\T_0,y_0;\varepsilon)-\varphi^-_2(\tau^-(\T_0,y_0;\varepsilon),\T_0+\sigma,y_0;\varepsilon)).
\end{align*}

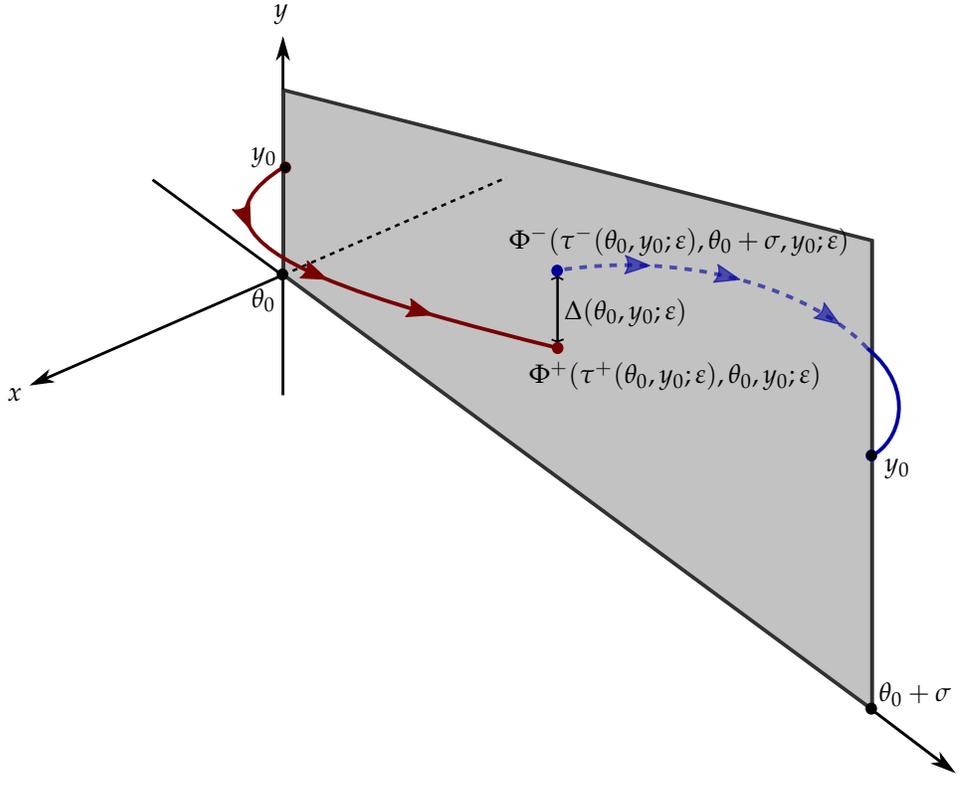
\begin{figure}[!htb]\label{displacement}
	\begin{center}
	\begin{overpic}[scale=0.9,unit=1mm]{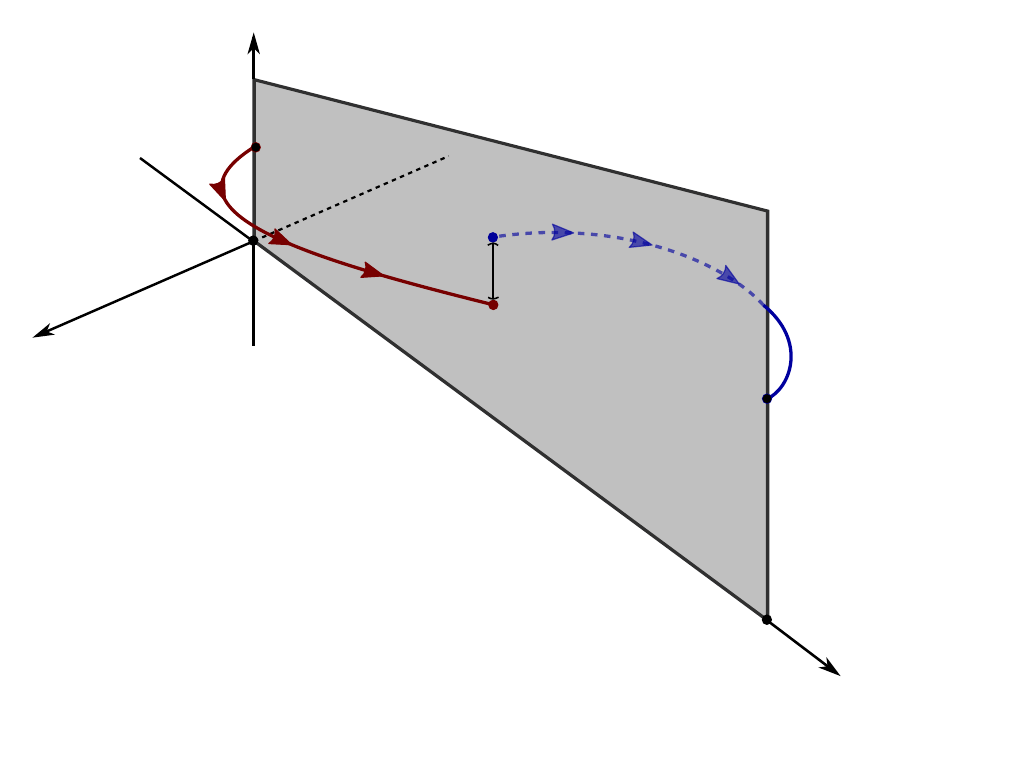}
			\put(24,74.7){\footnotesize$y$}
			\put(22.1,62.5){{\footnotesize $y_0$}}
			\put(76,36){{\footnotesize $y_0$}}
			\put(22.1,50){{\footnotesize $\T_0$}}
			\put(75.5,16.4){{\footnotesize $\T_0+\sigma$}}
			\put(1.3,42){\footnotesize$x$}
			\put(82.5,9){\footnotesize$t$}
			\put(48.7,49){\footnotesize$\Delta(\T_0,y_0;\e)$}
			\put(44,55){{\footnotesize $\Phi^-(\tau^-(\T_0,y_0;\varepsilon),\T_0+\sigma,y_0;\varepsilon)$}}
			\put(45.7,43.5){{\footnotesize $\Phi^+(\tau^+(\T_0,y_0;\varepsilon),\T_0,y_0;\varepsilon)$}}
		\end{overpic}
	\end{center}
\vspace{-1,3cm}
	\caption{Representation of the points $\Phi^-(\tau^-(\T_0,y_0;\varepsilon),\T_0+\sigma,y_0;\varepsilon)$ and $\Phi^+(\tau^+(\T_0,y_0;\varepsilon),\T_0,y_0;\varepsilon)$, which originate the displacement function $\Delta$.}\label{fig}
\end{figure}

In what follows, we provide preliminary results concerning the main ingredients constituting the displacement function $ 	\Delta(\T_0,y_0;\e) $.

\subsection{Preliminary results} 
This section is dedicated to presenting preliminary results regarding the solutions of \eqref{sis:s2} and the time functions $ \tau^+(\T_0,y_0;\varepsilon) $ and $ \tau^+(\T_0,y_0;\varepsilon) $ mentioned earlier. We begin by providing a result concerning the behavior of the solutions of \eqref{sis:s2} as $ \e $ approaches to zero.

\begin{proposition} For sufficiently small $ |\e |$, the function   $ \varphi^{\pm}(t,\T_0,y_0;\e) $ writes as 
\begin{equation}\label{eq:phipm}
	\varphi^{\pm}(t,\T_0,y_0;\varepsilon)=\Gamma^{\pm}(t,y_0)+\varepsilon\psi^{\pm}(t,\T_0,y_0)+\CO(\varepsilon^2),
\end{equation}
where $\Gamma^{+}(t,y_0)  $ and $\Gamma^{-}(t,y_0)  $ are the functions given in \eqref{solucaoGama} and \eqref{rel}, respectively, and
\begin{equation}\label{eq:psi}
	\psi^{\pm}(t,\T_0,y_0)=e^{At}\int_{0}^{t}e^{-As}F(s+\T_0,\Gamma^{\pm}(s,y_0)) \d s.
\end{equation}
\end{proposition}

\begin{proof}
Since $\varphi^{\pm}(t,\T_0,y_0;\varepsilon)$ is the solution to the Cauchy problem \eqref{sis:s2}, then it must satisfy the integral equation
\begin{align*} 
\varphi^{\pm}(t,\T_0,y_0;\varepsilon)  = \x_0
+\int_{0}^{t}\left[A\varphi^{\pm}(s,\T_0,y_0;\varepsilon)\mp\ab+\varepsilon F(s+\T_0,\varphi^{\pm}(s,\T_0,y_0;\varepsilon)\right] \d s ,
\end{align*}
which, by expanding in Taylor series around $ \e=0 $, gives us 
\begin{align*}
	\varphi^{\pm}(t,\T_0,y_0;\varepsilon) &=(0,y_0)+\int_{0}^{t}\left[A\Gamma^{\pm}(s,y_0)\mp\ab\right]\d s+\varepsilon\int_{0}^{t}\left[ A\psi^{\pm}(s,\T_0,y_0) \right.\\
& \hspace{2.5cm}\left.+F(s+\T_0,\Gamma^{\pm}(s,y_0))\right]\d s + \CO(\varepsilon^2).
\end{align*}
Then, taking into account the expression for $ \varphi^{\pm}(t,\T_0,y_0;\varepsilon)   $ in \eqref{eq:phipm} and the computations above, we have that
\begin{equation*}\label{PC1}
\psi^{\pm}(t,\T_0,y_0)= \int_{0}^{t} \left[A\psi^{\pm}(s,\T_0,y_0) +F(s+\T_0,\Gamma^{\pm}(s,y_0))\right]\d s,
\end{equation*}
which implies that $\psi^{\pm}$ is the solution to the Cauchy problem
\[
\begin{cases}
	\dot{\x}=A\x+F(t+\T_0,\Gammapm(s,y_0)),\\
	\x(0)=(0,0)	.
\end{cases}
\]
Then, the general formula for solutions of linear differential equations yields relationship \eqref{eq:psi}.\end{proof}

In what follows, we describe the behavior of the time functions $\tau^{+}(\T,y;\e)$ and $\tau^{-}(\T,y;\e)$ satisfying \eqref{ss}.

\begin{proposition}\label{PP}
	Let $\taupm(\T_0,y_0;\e)$ be the time satisfying equation \eqref{ss}. Then, for sufficiently small $|\e| $ and $ y_0>0 $ we have
	\[
	\taupm(\T_0,y_0;\e)=\taupm_0(y_0)+\e \taupm_1(\T_0,y_0)+\CO(\varepsilon^2),
	\]
	with
	\[
	\taupm_1(\T_0,y_0)=\frac{\psi^{\pm}_1(\taupm_0(y_0),\T_0,y_0)}{y_{0}}.
	\]
\end{proposition}

\begin{proof}
By expanding \eqref{ss} in Taylor series around $\e=0$, we have
 \[
\varepsilon\left(\frac{\partial\Gammapm_1}{\partial t}(\taupm_0(y_0),y_0)\taupm_1(\T_0,y_0)+\psi^{\pm}_1(\taupm_0(y_0),\T_0,y_0)\right)+\CO(\varepsilon^2)=0,
\]
which implies that
\begin{equation}\label{eqs}
\frac{\partial\Gammapm_1}{\partial t}(\taupm_0(y_0),y_0)\taupm_1(\T_0,y_0)+\psi^{\pm}_1(\taupm_0(y_0),\T_0,y_0)=0.
\end{equation}
Then equation~\eqref{eqs} together with \eqref{eq:Gama2} conclude the proof of the proposition. \end{proof}

The following result plays an important role in describing the behaviour of $ \varphipm_2 $ around $ \e=0 $. It is important to mention that, in our context, we identify vectors with column matrices.

\begin{proposition}\label{PropA}
	Let us consider $v_1(t)$ and $v_2(t)$ as the lines of the matrix $  e^{A t} $. Then for every $ y_0>0 $, the following identity holds
	\begin{equation*}
		\alpha v_1(\tau_0(y_0))-y_0 v_2(\tau_0(y_0))=(\alpha \quad y_0),
	\end{equation*}
where $ \tau_0(y_0) $ is the half--period function defined in \eqref{tauzero}.
\end{proposition}

\begin{proof}
	We define the auxiliary matrix-valued function 	
	\begin{equation*}\label{eq:beta}
		\beta(t)=\left(\alpha-\eta \Gamma_1^+(t,y_0) \quad \Gamma_2^+(t,y_0) \right)\cdot e^{A t},
	\end{equation*}
which is continuously differentiable for every $ t\in\R $. By differentiating $ \beta(t) $, we have that 
\begin{align*}
	\beta'(t)&=\left(-\eta \Gamma_2^+(t,y_0) \quad \eta \Gamma_1^+(t,y_0)-\alpha\right) \cdot e^{At}
+\left(\alpha-\eta \Gamma_1^+(t,y_0) \quad \Gamma_2^+(t,y_0) \right) \cdot A e^{At}\\
&=\left(-\eta \Gamma_2^+(t,y_0) \quad \eta \Gamma_1^+(t,y_0)-\alpha\right) \cdot e^{At}+\left(\eta \Gamma_2^+(t,y_0)\quad\alpha-\eta \Gamma_1^+(t,y_0) \right) \cdot e^{At}\\
&=(0 \quad 0) \cdot e^{At}\\
&=(0\quad 0).
\end{align*}
Computations above imply that $ \beta(t) $ is a constant function in each one of its entries, that is, 
\[
   \beta(t) =\beta(0)=(\alpha \quad y_0)\cdot \text{Id}=(\alpha \quad y_0)
\]
 for every $ t\in\R $. In particular,
\[
 		\alpha v_1(\tau_0(y_0))-y_0 v_2(\tau_0(y_0))=\beta(\tau_0(y_0))=(\alpha\quad y_0),
 \]
 and this concludes the proof of the proposition.\end{proof}

In the following discussion, we provide key relationships for the fundamental components that appear in the expressions of $\varphi^{+}_2(\tau^{+}(\T_0,y_0;\e),\T_0,y_0;\e)$ and $\varphi^{-}_2(\tau^{-}(\T_0,y_0;\e),\T_0+2\sigma,y_0;\e)$ for sufficiently small values of $|\e|$. We start by formulating a more detailed expression for $ \psi^+(t,\T_0,y_0) $. By taking relation \eqref{eq:psi} into account, we have
\begin{align}
	\label{psi1}\psi^+(t,\T_0,y_0)&=\begin{pmatrix}
		\psi_1^+(t,\T_0,y_0) \\
		\psi_2^+(t,\T_0,y_0)
	\end{pmatrix}=e^{At}\int_{0}^{t}e^{-As}F(s+\T_0,\Gamma^+(s,y_0))\d s\\
	&=\begin{pmatrix}
		\langle v_1(t)^{\top} ,I^+(t,\T_0,y_0) \rangle \\
		\langle v_2(t)^{\top} ,I^+(t,\T_0,y_0) \rangle
	\end{pmatrix}\nonumber,
\end{align}
where  
\begin{equation}\label{I+}
	I^+(t,\T_0,y_0):=\int_{0}^{t}e^{-As}F(s+\T_0,\Gamma^+(s,y_0)) \d s.
\end{equation}
This remark leads us to the following result.

\begin{lemma}\label{lem:phi2p}
	For sufficiently small $ |\e |$, the function $\varphi^{+}_2(\tau^{+}(\T_0,y_0;\e),\T_0,y_0;\e)$ writes as
	\[
	\varphi^{+}_2(\tau^{+}(\T_0,y_0;\e),\T_0,y_0;\e)=-y_0-\frac{\e}{y_0}
\langle(\alpha,  y_0),I^+(\tau_0(y_0),\T_0,y_0)\rangle +\CO(\varepsilon^2).
	\]
\end{lemma}

\begin{proof}
	By expanding $\varphi^{+}_2(\tau^{+}(\T_0,y_0;\e),\T_0,y_0;\e)$ in Taylor series around $ \e=0 $, we have that
	\[
	\varphi^{+}_2(\tau^{+}(\T_0,y_0;\e),\T_0,y_0;\e)= \Gamma_2^+(\tau_0(y_0),y_0)+\e\xi^+(\T_0,y_0)+\CO(\e^{2}),
	\]
	where 
	\begin{equation*}\label{eq:xip}
\begin{aligned}
			\xi^+(\T_0,y_0)&=\frac{\d}{\d  t}\Gamma_2^+(\tau_0(y_0),y_0)\tau^{+}_1(\T_0,y_0)+\psi_2^+(\tau_0(y_0),\T_0,y_0)\\
			&=\left(\T\Gamma_1^+(\tau_0(y_0),y_0)-\alpha\right)\frac{	\psi_1^+(\tau_0(y_0),\T_0,y_0)}{y_0}+\psi_2^+(\tau_0(y_0),\T_0,y_0)\\
			&=-\frac{1}{y_0}\left(\alpha\psi_1^+(\tau_0(y_0),\T_0,y_0)-y_0\psi_2^+(\tau_0(y_0),\T_0,y_0) \right).
\end{aligned}
	\end{equation*}
Thus, by considering the relationship \eqref{psi1}, the function $ 	\xi^+(\T_0,y_0) $ can be rewritten as follows
\[
\begin{aligned}
		\xi^+(\T_0,y_0)&=-\frac{1}{y_0}\left(\alpha	\langle v_1(\tau_0(y_0))^{\top} ,I^+(\tau_0(y_0),\T_0,y_0) \rangle -y_0\langle v_2(\tau_0(y_0))^{\top} ,I^+(\tau_0(y_0),\T_0,y_0) \rangle  \right)\\
		&=-\frac{1}{y_0}\left(\langle\alpha v_1(\tau_0(y_0))^{\top}-y_0v_2(\tau_0(y_0))^{\top} ,I^+(\tau_0(y_0),\T_0,y_0) \rangle \right)\\
		&=-\frac{1}{y_0}\left(\langle(\alpha,y_0),I^+(\tau_0(y_0),\T_0,y_0)\rangle\right),
\end{aligned}
\]
where the last equality above is obtained after Proposition \ref{PropA}. Then, taking into account \eqref{eq:y0}, the proof of the lemma is completed.\end{proof}

In order to obtain analogous results as those achieved for $\varphi^{+}_2(\tau^{+}(\T_0,y_0;\e),\T_0,y_0;\e)$ in the context of the function $\varphi^{-}_2(\tau^{-}(\T_0,y_0;\e),\T_0+\sigma,y_0;\e)$, we follow the previously outlined procedure. Then, by taking into account that $F(t,x,y)$ is a $\sigma$-periodic function in $ t $ and relationship \eqref{eq:eA}, we notice that
\[
\begin{aligned}
	\label{psi2}\psi^-(-t,\T_0+\sigma,y_0)&=\begin{pmatrix}
		\psi_1^-(-t,\T_0+\sigma,y_0) \\
		\psi_2^-(-t,\T_0+\sigma,y_0)
	\end{pmatrix}\\
	&=e^{-At}\int_{0}^{-t}e^{-As}F(s+\T_0+\sigma,\Gamma^-(s,y_0))\d s\nonumber\\
	&=-Re^{At}R\int_{0}^{t}e^{As}RRF(-s+\T_0,R\Gamma^+(s,y_0))\d s\nonumber\\
	&=-R\begin{pmatrix}
		\langle v_1(t)^{\top} ,\int_{0}^{t}e^{-As}RF(-s+\T_0,R\Gamma^+(s,y_0)) \d s \rangle \\
		\langle v_2(t)^{\top} ,\int_{0}^{t}e^{-As}RF(-s+\T_0,R\Gamma^+(s,y_0))\d s \rangle
	\end{pmatrix}\nonumber\\
	&=\begin{pmatrix}
		\langle v_1(t)^{\top} ,I^-(t,\T_0,y_0) \rangle \\
		\langle- v_2(t)^{\top} ,I^-(t,\T_0,y_0) \rangle
	\end{pmatrix},\nonumber
\end{aligned}
\]
where 
\begin{equation}\label{I-}
I^-(t,\T_0,y_0):=	\int_{0}^{t}e^{-As}RF(-s+\T_0,R\Gamma^+(s,y_0)) \d s.
\end{equation}

Similarly proceeding as in Lemma \ref{lem:phi2p}, we provide the following result concerning the behavior of  $\varphi^{-}_2(\tau^{-}(\T_0,y_0;\e),\T_0+\sigma,y_0;\e)$ around $ \e=0 $.

\begin{lemma}\label{lem:phi2N}
	For sufficiently small $| \e|  $, the function $\varphi^{-}_2(\tau^{-}(\T_0,y_0;\e),\T_0+\sigma,y_0;\e)$ writes as
	\[
\varphi^{-}_2(\tau^{-}(\T_0,y_0;\e),\T_0+\sigma,y_0;\e)=-y_0+\frac{\e}{y_0}
	\left(\langle(\alpha, y_0),I^-(\tau_0(y_0),\T_0,y_0)\rangle \right)+\CO(\varepsilon^2).
	\]
\end{lemma}

Before we proceed with the proof of Theorem \ref{ta}, let us perform some essential computations which will play an important role in deriving the desired Melnikov-like function. Let $ I^+(t,\T_0,y_0)$ and $ I^-(t,\T_0,y_0) $ be the integrals defined in \eqref{I+} and \eqref{I-}, respectively. We remind that $ F(t,x,y)=(0,f(t,x,y))$. Taking $u(t)=(u_1(t),u_2(t))$ to be the second column of $ e^{-At} $, that is, 
\begin{equation*}\label{u}
u(t)=\left\{\begin{aligned}
\left(-\frac{\sinh(t\sqrt{\eta})}{\sqrt{\eta}},\cosh(t\sqrt{\eta})\right)\quad &\text{if} \quad \eta>0,\\
(-t,1)\quad &\text{if} \quad \eta=0,\\
\left(-\frac{\sin(t\omega)}{\omega},\cos(t\omega)\right)\quad &\text{if} \quad \eta<0,
\end{aligned}\right.
\end{equation*}
with $ \omega=\sqrt{-\eta} $, it follows that 
\begin{align}\label{somaI+I-}
	I^+(t,\T_0,y_0)+I^-(t,\T_0,y_0)&=\int_{0}^{t}e^{-As}\begin{pmatrix}
		0\\
		g(s,\T_0,y_0)
	\end{pmatrix}\d s	\\
&=\begin{pmatrix}
		\int_{0}^{t}u_1(s)g(s,\T_0,y_0)\d s \\
		\int_{0}^{t}u_2(s)g(s,\T_0,y_0)\d s	
	\end{pmatrix},\nonumber
\end{align}
where we are defining $g(s,\T_0,y_0):= f(\T_0+s,\Gamma^+(s,y_0))+f(\T_0-s,R\Gamma^+(s,y_0))$. Notice that $ g(s,\T_0,y_0) $ is $ \sigma $-periodic in $ \T_0 $.

\subsection{Conclusion of the proof of Theorem A}
The task of obtaining a point $(\T_0,y_0;\e)$ that directly makes the function $\Delta$ vanish is quite challenging. Thus, in our approach, we proceed with a Melnikov--like method, which basically consists in computing the Taylor expansion of $ \Delta(\cdot,\cdot;\e) =0$ around $ \e=0 $ up to order 1 and solving the resulting expression. In this direction, we start by examining the first component of the function $ \Delta $, which, after its Taylor expansion around $ \e=0 $, is given by
\begin{align*}
	\Delta_1(\T_0,y_0;\e)	=\tau^+(\T_0,y_0;\e)-\tau^-(\T_0,y_0;\e)-\sigma=2\tau_0(y_0)-\sigma+\CO(\e).
\end{align*}

Let $ i \in\{1,4,7,9\} $ be fixed such that the parameters $ \alpha $ and $ \eta $ satisfy condition \textbf{(Ci)} and $ \sigma/2\in \mathcal{I} _i$. Since $\tau_0 $ is a bijection between $ \mathcal{D}_i $ and $ \mathcal{I} _i $, there exist $ y_0^{*}\in\mathcal{D}_i $ such that $ \tau_0(y_0^{*})=\sigma/2$. Additionally, as discussed in Section \ref{unpanalise}, $\frac{\partial \Delta_1}{\partial y_0}(\T_0,y^*_0;0)=2\tau_0'(y_0^{*})\neq0$ for every $ \T_0\in\s_{\sigma} $. Therefore, from the compactness and the Implicit Function Theorem, there exist $\e_1> 0$, $\delta_1>0$, and a unique $\CC^1$-function ${\overline{y}:\s_{\sigma}\times(-\e_1,\varepsilon_1)\to (y^*_0-\delta_1,y^*_0+\delta_1)}$ such that $\overline{y}(\T_0,0)=y^*_0$ and $ {\Delta_1(\T_0,\overline{y}(\T_0,\e);\e)=0}$, for every $\e\in (-\e_1,\e_1)$ and every $ \T_0\in \s_{\sigma} $.

By substituting $ \overline{y}(\T_0,\e) $ into $ \Delta_3(\T_0,y_0;\e) $, and taking into account Lemmas \ref{lem:phi2p} and \ref{lem:phi2N}, we have, for sufficiently small $ |\e| $,
\[
\Delta_3(\T_0,\overline{y}(\T_0,\e);\e) =- \frac{2\e}{y_0^{*}} 	\left(\left\langle(\alpha, y_0^{*}),I^+\left(\frac{\sigma}{2},\T_0,y_0^{*}\right)+I^-\left(\frac{\sigma}{2},\T_0,y_0^{*}\right)\right\rangle \right)+\CO(\e^{2}),
\]
where $ I^{+}(\sigma,\T_0,y_0^{*}) $ and $ I^{-}(\sigma,\T_0,y_0^{*}) $ are the integrals defined in \eqref{I+} and \eqref{I-}, respectively. We can then define the function, for $ |\e| $ sufficiently small, 
\[
\tilde{\Delta}_3(\T_0;\e):=-\frac{y_0^{*}}{2\e }\Delta_3(\T_0,\overline{y}(\T_0,\varepsilon);\e),
\]	
which, after being expanded in Taylor series around $\e=0 $, gives us
\[
\tilde{\Delta}_3(\T_0;\e)=M(\T_0)+\CO(\e),
\]
with
\[
M(\T_0) =\left\langle(\alpha, y_0^{*}),I^+\left(\frac{\sigma}{2},\T_0,y_0^{*}\right)+I^-\left(\frac{\sigma}{2},\T_0,y_0^{*}\right)\right\rangle.
\]
The identity \eqref{somaI+I-} and the fact that $ y_0^{*}=v(\sigma/2) $, allow us to rewritten $ M(\T_0)  $ as follows
\begin{equation*}
M(\T_0)=\int_{0}^{\frac{\sigma}{2}}\left\langle\left(\alpha, v\left(\frac{\sigma}{2}\right)\right),u(s)\right\rangle g\left(s,\T_0,v\left(\frac{\sigma}{2}\right)\right)\d s,
\end{equation*}
and this lead us to the expression stated in \eqref{eq:M}, with the auxiliary function \linebreak $ U(t,\sigma/2)=\langle(\alpha, v(\sigma/2)),u(t)\rangle  $ expressed in \eqref{eq:U}. Notice that the $ \sigma $-periodicity of $ g\left(s,\T_0,v\left({\sigma}/{2}\right)\right) $ in $ \T_0 $ implies that $ M $ is $ \sigma $-periodic, which enables us to restrict our analysis to the interval $ [0,\sigma] $. Now suppose that  $\phi^*\in[0, \sigma]$ is such that $M(\phi^*)=0$ and $M'(\phi^*)\neq0$. Then, by the Implicit Function Theorem, there exist $0< \ov \e <\e_1 $  and a branch $\overline{\phi}(\e)$ of simple zeros of $ M $ satisfying $\overline{\phi}(0)=\phi^*$ and $M(\overline{\phi}(\e))=\tilde{\Delta}_3(\overline{\phi}(\e);\e)=0$, for every $ \e\in(-\ov\e,\ov\e) $. Back to the solution of the differential system \eqref{eq:ext}, we have that $ \Phi(\tau,\overline{\phi}(\e),\overline{y}(\overline{\phi}(\e),\e);\varepsilon) $ is a $ \sigma $-periodic solution of \eqref{eq:ext}, whenever $\e\in(-\ov\e,\ov\e)$.

Notice that, by defining 
\[
x_{\e}(t):=\Phi_2(t-\overline{\phi}(\e),\overline{\phi}(\e),\overline{y}(\overline{\phi}(\e),\e);\e),
\]
and taking into account that $ \dot{x}_{\e}(t)= \Phi_3(t-\overline{\phi}(\e),\overline{\phi}(\e),\overline{y}(\overline{\phi}(\e),\e);\e)$, where $\Phi_2  $ and $\Phi_3  $ are the second and third components of $ \Phi$, respectively, we have that 
\[
x_{0}(\phi^*)=\Phi_2(0,\phi^*,y_0^{*} ;0)=0 \quad \text{and} \quad  \dot{x}_{0}(\phi^*)=\Phi_3(0,\phi^*,y_0^{*} ;0)=y_0^{*}=v\left(\frac{\sigma}{2}\right).
\]
It concludes the proof of Theorem \ref{ta}.

\section{Declarations}
\subsection*{Competing interests} On behalf of all authors, the corresponding author states that there is no conflict of interest.
\subsection*{Funding} DDN is partially supported by S\~{a}o Paulo Research Foundation (FAPESP) grants 2022/09633-5,  2019/10269-3, and 2018/13481-0, and by Conselho Nacional de Desenvolvimento Cient\'{i}fico e Tecnol\'{o}gico (CNPq) grant 309110/2021-1. LVMFS is partially supported by S\~{a}o Paulo Research Foundation (FAPESP) grant 2018/22398-0.

\bibliographystyle{acm}
\bibliography{references}
\end{document}